\documentclass[12pt]{article}
\usepackage{graphicx}
\usepackage{amsfonts,amsmath,amsthm,amssymb}
\usepackage{mathdots,multirow,bigdelim}
\usepackage{charter} 

{\bfseries}{\rmfamily}
\newtheorem{algorithm*}{Algorithm}{\bfseries}{\rmfamily}

\newtheorem{theorem}{Theorem}

\theoremstyle{definition}
\newtheorem{proposition*}{Proposition}
\newtheorem{definition}[theorem]{Definition}

\theoremstyle{remark}

\newcommand{\bt}{\begin{tabular}}
\newcommand{\et}{\end{tabular}}
\newcommand{\ba}{\begin{array}}
\newcommand{\ea}{\end{array}}
\newcommand{\bq}{\begin{eqnarray}}
\newcommand{\eq}{\end{eqnarray}}
\newcommand{\bqa}{\begin{eqnarray*}}
\newcommand{\eqa}{\end{eqnarray*}}
\newcommand{\ds}{\displaystyle}

\newcommand{\bs}{$\backslash$}

\def\nn{{\mathbb N}}
\def\ox{{\overline X}}
\def\xa{ }
\def\xb{ }
\def\xc{ }

\def\mca{\multicolumn{1}{c}{ }}
\def\mcb{\multicolumn{2}{c}{ }}
\def\mcc{\multicolumn{3}{c}{ }}
\def\mcd{\multicolumn{4}{c}{ }}

\begin{document}
\noindent{\large{\bf Friedman's ``Long Finite Sequences'':\\
The End of the Busy Beaver Contest
}}

\vspace*{5mm}
\noindent{Michael Vielhaber}\\
Hochschule Bremerhaven, FB2,
An der Karlstadt 8,  D--27568 Bremerhaven\\
{\tt vielhaber@gmail.com}

\vspace*{5mm}
\noindent{M\'onica del Pilar Canales Chac\'on}\\
MATEMATICVM, Valdivia, Chile\\
@matematicvm\\
{\tt monicadelpilar@gmail.com}

\vspace*{5mm}
\noindent{Sergio Jara Ceballos}\\
Facultad de Ingenier\'{\i}a, Universidad Austral de Chile, Valdivia, Chile\\
{\tt serjara.ing.mat@gmail.com}

\vspace*{1cm}

\begin{abstract}
  Harvey Friedman gives a comparatively short description of an ``unimaginably large'' number $n(3)$ , beyond {\it e.g.} the values
  $$  A(7,184)< A({7198},158386) < n(3)$$
  of Ackermann's function - but finite.
  
  We implement Friedman's combinatorial problem about subwords of words over a 3-letter alphabet on a family of Turing machines, which, starting on empty tape, run (more than) $n(3)$ steps, and then halt.
Examples include a (44,8) (symbol,state count) machine as well as a (276,2) and a (2,1840) one.

In total, there are at most 37022 non-trivial pairs $(n,m)$ with Busy Beaver values ${\tt BB(n,m)} < A(7198,158386).$

We give algorithms to map any $(|Q|,|E|)$ TM to another, where we can choose freely either $|Q'|\geq 2$ or $|E'|\geq 2$ (the case $|Q'|=2$ for empty initial tape is the tricky one). 

Given the size of $n(3)$ and the fact that these TMs are not {\it holdouts}, but assured to stop, Friedman's combinatorial problem provides a definite upper bound on what might ever be possible to achieve in the Busy Beaver contest.
We also treat $n(4)> A^{(A(187196))}(1)$.

{\bf Keywords:} Busy beaver, long finite sequences, Turing machine.
\end{abstract}

\subsection*{Introduction}

Harvey Friedman describes in \cite{F} the problem to decide for any alphabet $B$ of size $k\in\nn$ the length of a largest word $s\in B^*$, such that property 
$$(*) \nexists\ 1\leq i < j\leq n/2\colon  s^{(i)} :=s_is_{i+1}\dots s_{2i}\mbox{ is a subword of } s^{(j)} :=s_js_{j+1}\dots s_{2j}$$
is satisfied.

For  $k=1$ {\it i.e.} $B=\{1\}$, 111 satisfies $(*)$, since $s_2s_3s_4$ is not even defined, but $s=1111$, the only word in $B^4$, already violates $(*)$ for $i=1,j=2$.
Thus, $n(1)=3$.

Similarly, for $k=2$ and $B=\{1,2\}$, we find the word $s=11222111111$ of length 11 satisfying $(*)$, but all $2^{12}$ words of length 12 (and thus all larger ones) violate $(*)$ (see \cite[p.~3f.]{F}):
This gives us $n(1)=3, n(2)=11$.

How does it go on thereafter with $n(3), n(4),\dots$?
\\\\
Unexpected!
\\
Friedman \cite[Sect.~4]{F} first shows a lower bound of $n(3) > A(7,184)$  and in \cite[Sect.~6]{F} presents results by Dougherty yielding even 
$$n(3) > A({7198},158386).$$

This current lower bound of $A({7198},158386) $ is (by far!) larger than {\it e.g.}   $A(3,158386) = 2^{2^{\iddots^{2^2}}}$, with 158386 2's stacked exponentially.

One bit of $n(3)$ we know, though: The last one. Since all $n(k)$ are odd.

We describe Turing machines guaranteed to halt after more than $n(3)$ steps, starting with an empty tape for {\it all} non-trivial pairs $(n,m)\in \nn^2$ of (state, symbol) counts with only 37022 exceptions.
Thus the Busy Beaver contest is effectively a {\it finite} matter.

Section~1 covers the relevant work of Ackermann, Turing, Rad\'o and Friedman.
Section~2 describes the algorithm and an 8-symbol im\-plemen\-ta\-tion with 44 states.

In Sections~3 and 4, we show how to get either the symbol count or the state count down to 2, for {\it any} Turing machine.
Section~5 briefly treats the case $n(4)$ over 4 symbols and in Section~6, we obtain the main result:

All but at most 37022 non-trivial $(n,m)$ pairs have a {\tt BB(n,m)} value above $n(3)$, and at most 51671 {\tt BB(n,m)} lie below $n(4)$.


\section{Four Mathematicians and their Crucial Results}

\subsection{Wilhelm Ackermann, 1926: ``Ackermann~function''}

The (modified) Ackermann function \cite{Ackermann} used here  as defined in \cite{F} is $A(1,c) = 2c$, $A(f,1) = 2$, and recursively $A(f,c)=A({f-1},A(f,c-1))$, which is equivalent to
a $c$-fold nesting of $A(f-1,\cdot)$.

The parameter $f$ can be seen as defining a function family, while the counter $c$ is a pointer into this family.
Every sequence $A(f,\nn)$ is a sub\-sequen\-ce of the previous $A(f-1,\nn)$.
Increasing $f$ is what lets the values explode.

For a nice overview of {\it really} large numbers, even beyond $|\nn|=\aleph_0$, see {\tt https://sites.google.com/site/largenumbers/home}.

One easily obtains $A(1,c) = 2\cdot c,  A(2,c) = 2^c, 
A(3,c) = 2^{2^{\iddots^{2^2}}}$, with $c$ copies of 2 stacked onto each other.
Also, $A(f,1)=2, A(f,2)=4, \forall f.$

$A(4,3) 
= A(3,A(3,A(3,1)))
= A(3,A(3,2))
= A(3,4)
= 2^{2^{2^2}} = 65536$.

$A(4,4) 
= A(3,A(4,3))
= A(3,65536)= 2^{2^{\iddots^{2^2}}}$, $65536-1$ exponents.

Thus $A(4,5) = A(3,A(4,4))$ will be a tower of 2's, whose height is described by a tower of $65536$ 2's, ``and so on''.

$A(5,3) 
= A(4,A(5,2))
= A(4,4)$ as above, while

$A(5,4) 
= A(4,A(5,3))
= A(4,A(4,4))$
is an $A(4,4)$-fold  iterated evalua\-tion of $A(3,\cdot{})$ -- we have no idea of its value, and not even a means to visualize that number.

Friedman \cite[p.~106]{F} calls $A(5,5)$ an ``unimaginably large number''.
We have nothing to add.

\subsection{Alan M.~Turing, 1936: ``Turing Machine''}

We use Turing's invention \cite{Turing} with the following modifications:\\
-- no output tape or  F (figures) cells\\
-- bi-infinite tape, no (left) end markers needed or provided\\

There is a finite set $Q$ of states with $n := |Q|$, the fixed tape alphabet $E=\{0,1\}$ for Rad\'o's original problem, or any larger, but finite alphabet $E$  with $m := |E|$.

The three relevant functions are\\
$Q^+\colon Q\times E\to Q$ for the next state,\\
$E^+\colon Q\times E\to E$ for the new symbol to be written onto the tape, and\\
$D^+\colon Q\times E\to \{R,L\}$ for the movement of the tape's head.

\subsection{Rad\'o Tibor, 1962: ``Busy Beaver''}

First some nomenclature:

1. A Busy Beaver is {\it any} Turing machine, which, starting with an empty tape, eventually halts.
{\it When} it halts, is {\it not} relevant for being a Busy Beaver (see Rad\'o \cite{Rado}, Michel \cite[p.~4]{Michel}, or Green \cite{Green}).
The generalized form allows for larger symbol sets (tape alphabets) than the original $B=\{0,1\}$ of Rad\'o.

2. The Busy Beaver Contest suggested by Rad\'o consists in providing a TM configuation's QED$^+$ values and a purported halting time T. If said configuration on empty tape stops after exactly T steps, the entry is valid.

3. The current Busy Beaver Champion of its class $(n,m) = (|Q|,|E|)$ is the halting machine with -- up to the time -- largest T.

4. The Busy Beaver function {\tt BB(n)} = {\tt BB(n,2)} or {\tt BB(n,m)} is the time T of the {\tt BB} champion for the class {\tt (n,m)} --- provided it is proved by any means, {\it e.g.} exhaustive search, that no other machine in that class exists that might halt on empty tape and do so after more than $T$ steps.
Otherwise the T of the current champion is a lower bound for {\tt BB(n,m)}.

\begin{figure}[b!]
  \centering
    \bt{c|rrrrrrr}
    7&7&\\
    6&6&A(3,15)\\
    5&5&A(2,25)\\
    4&4&107&A(2,$2^{\lceil 15.5\rceil}$)\\
    3&3&21&A(2,56)&A(2,$2^{\lceil 15.4\rceil}$)\\
    2&2&4&38&A(2,21)&A(2,$2^{11}$)&A(2,$2^{15}$)\\
    1&1&1&1&1&1&1&1\\
    \cline{1-8}
    $n /  m$&1&2&3&4&5&6&7\\
    \et
    \caption{Known values or lower bounds for {\tt BB(n,m)}}\label{fig:BB}
\end{figure}

Known results are given in Figure~\ref{fig:BB}, see Michel \cite{Michel}.
Numbers are exact figures, values of the Ackermann function are lower bounds (with $A(2,2^{16})=A(3,4)$).

\subsection{Harvey Friedman, 2001: ``Long Finite Sequences''}

We shall show in this paper that there are at most 37022  non trivial ($|Q|,|E|\geq 2$) Busy Beaver contests, since all other pairs lead to a lower bound of $A({7198},158386)<{\tt BB(n,m)}$ steps for a halting configuration implementing Friedman's combinatorial problem ``{\it Long Finite Sequences}''.

Friedman considers finite words over a $k$ letter alphabet, in particular $k=3$ and $B=\{1,2,3\}$.
His crucial definition $(*)$ is:
\begin{quote}
A word $s=s_1s_2\dots s_n$ from $B^n$ of length $n$ satisfies property $(*)$ whenever the set of subwords $s^{(i)}=s_is_{i-1+1}\dots s_{2i}, 1\leq i\leq n/2$, that is $s_1s_2, s_2s_3s_4,\dots, (s_i\dots s_{2i}), \dots, (s_{n/2}\dots s_n)$, does \emph{not} contain two words with $i<j$ such that $s^{(i)}$ is a subword of  $s^{(j)}$.

A word $a_1\dots a_m$ is called a subword of $b_1\dots b_n$ whenever there are indices $1\leq \iota_1<\iota_2<\dots<\iota_m\leq n$ with $a_k=b_{\iota_k}$.
\end{quote}
For every $k\in\nn$, let $n(k)$ be the length of a largest word from $\{1,\dots,k\}^*$ satisfying $(*)$.
One easily verifies $n(1)=3$.

Friedman shows that $12221111111\in\{1,2\}^*$ satisfies $(*)$, but no larger word over two letters, and thus $n(2)=11$.

All $n(k)$ are odd, by the way.
The last odd indexed letter will not generate a new subword and thus the status of $(*)$ does not change.

After $n(1)=3$ and $n(2)=11$, we have quite a jump:

\begin{theorem} {\rm (Friedman\cite[Theorem 4.7]{F})}
  
$$n(3) > A(7,184).$$

\end{theorem}
\begin{proof} See Theorem 4.7 in \cite{F}.\end{proof}

This paper will implement the search for the first word not satisfying $(*)$ over $B=\{1,2,3\}$ and this search will take more than $n(3)$ steps -- way above the ``incomprehensibly large number'' $A(5,5)$ -- and then halt.

Dougherty even obtains the following fantastically large bound:
\begin{theorem}
{\rm (Friedman\cite[Theorem 6.9]{F})}
$$n(3) > A({7198},158386).$$
\end{theorem}
\begin{proof} See Theorem 6.9 in \cite{F}.\end{proof}

Friedman furthermore conjectures \cite[p.~7]{F2} the upper bound
$$n(3) < A(A(5,5),A(5,5)).$$


\section{The Algorithm and an 8-Symbol Implementation}
\begin{figure}[b!]
  \centering
    \bt{rl}
    1.&Copy $s$ from II to III, N/2 times, separated by $+$.\\
    2.&In III, cut away initial triangle 
    $\varepsilon, s_1,s_1s_2,\dots, s_1\cdots s_{i-1},\dots , s_1\cdots s_{N/2-1}$.\\
    3.&In III, cut away double trailing triangle 
    $s_3\cdots s_{N},s_5\cdots s_{N},\dots,s_{N-1}s_{N},\varepsilon$,\\
    &leaving $s_1s_2,s_2s_3s_4,\dots, s^{(i)},\dots,s^{(N/2)}$ on the tape as III.\\
    4.&Remove the $(i-1)$ patterns $s^{(1)},\dots, s^{(i-1)}$.\\
    5.&Check $s^{(i)},\dots, s^{(N/2)}$ for subword match, property $(*)$.\\
    6.&Clear segment III to X$\varepsilon$Y$^\omega$. IF match in l.~5, GOTO 7 ELSE GOTO 8.\\
    7.&$s++;$ IF $s=(1)^{N+1}$ HALT ELSE $i:=0,l:=0$. GOTO 1.\\
    8.&$i:=0; lmax++;$ IF $lmax < N/2$ GOTO 1 ELSE GOTO 9.\\
    9.&$(N/2)++$; $s:=(1)^N$; GOTO 1.\\
    \et
    \caption{Algorithm: Long Finite Sequences}\label{Alg}  
\end{figure}
We shall start with the symbol set $E=\{Y,X,1,2,3,-,\$,+\}$, where $Y$ is the blank.

The active part of the tape is divided into 3 segments:

I Two unary counters $0\leq i \leq imax=N/2$ and $0\leq l\leq lmax \leq N/2$.

II The current word $s\in B^N$ in the form $s_1s_2\dots s_N\in \{1,2,3\}^N$ or certain symbols replaced by their prime equivalent, with $-\equiv 1', \$\equiv 2', +\equiv 3'$. 

III $N/2$ copies of $s$, separated by '+'s, which are then trimmed to\\
$s^{(1)} =(s_1s_2), \dots, s^{(i)} =(s_i..s_{2i}),\dots, s^{(N/2)} =(s_{N/2}..s_{N})$.

Segment bounds are given by the markers Y,Y,X,Y that is the whole tape has a structure like 
$${}^{\omega}{\rm Y}\cdots {\rm I}\cdots {\rm Y}\cdots {\rm II}\cdots {\rm X}\cdots {\rm III}\cdots {\rm Y}^\omega$$

This marker sequence, YYXY, has the advantage  -- defining the tape's blank symbol as Y -- that both ends are immersed within the 
${}^\omega$ Y...Y$^\omega$ of the bi-infinite tape and are thus automatically correct upon extension of segments.

The algorithm consists of 9 lines as given in Figure~\ref{Alg}.

States in $Q$ are named `ql-c', where $l \in \{1,\dots,9\}$ refers to the line of the algorithm, and $c\in\nn_0$ is just a counter within the line.

On pages 8--12, we describe the 44 states implementing the algorithm.
For each program line, we indicate the segment dealt with (I, II, or III) and the starting position by @I, l.h.s.~end of I to III@, r.h.s.~end of III.

Each entry gives the relevant symbols from $E=\{Y,X,-,\$,+,1,2,3\}$ in the first line; an asterisk '*' stands for all symbols not yet mentioned before.
An entry like $*\backslash 3,X$ stands for all other symbols, where 3 and X are not actually used (think $Q^+($q1-1,3)=$Q^+($q1-1,X) = ERROR).
Symbols that are omitted behave like the 3,X in  $*\backslash 3,X$:
They will not appear with this state.
The 2nd line of each entry gives the new symbol $E^+$ (`=' meaning no change in symbol) and the direction $D^+$, left (L) or right (R), of the tape head.
The third line is the next state $Q^+$.

We start with $i=l=0,$ I = II = III = $\varepsilon$ in line 3.

The TM starts in state q1-4, initializes a Y to X, then moves on to {q3-1},\linebreak q4-1, q4-2, q5-0, q5-1, q6-2, q8-0, q8-1, q8-2, q9-1 (tape is empty, all Y, except that one X), and in line 9 we start to increase the support from empty to length $N/2=1$ in Segment I, and $s=--\equiv 1'1'$ in Segment II.

State q7-1 is the finishing state, going into HALT with Y.

States q9-3 and q9-5 have been merged into q3-3 and q1-9, respectively, to save on state count.
In q9-3 and q9-5 we only deal with symbols ``--'' and ``Y'', which are absent in (the original) q3-3 and q1-9.

1,2,3 stand for themselves, also $-$,\$,+ stand for 1',2',3' in Segment II.

\begin{figure}
  \centering
  \bt{l|cccccc}
  \multicolumn{7}{l}{Definition:}\\
  \cline{1-7}
  Symbol&--&\$&+&1&2&3\\
  $l$&--&0&1&--&0&1\\
  $i$&0&0&0&1&1&1\\
\et
\bt{c}
\ \ 
\et
\bt{l|cccccccccc}
\multicolumn{10}{l}{Example:}\\
\cline{1-10}
Symbol&{--}&{--}&1&2&2&3&3&3&3\\
$l$&{--}&{--}&{--}&0&0&1&1&1&1\\
$i$&0&0&1&1&1&1&1&1&1\\
\et
\caption{Counters $l$ and $i$ and their encoding \label{fig:I_N}}.
\end {figure}

The two counters $l,i$ in Segment I are coded as in Figure~\ref{fig:I_N}.
The example value is $i = 7$ with a maximum of $imax=9$, and $l = 4$ with a maximum of $lmax=6$ (the 0s can change to 1s, the '--' can not).

We get $|E|$ down to 7 by removing the marker X.
This requires to do a double scan 
$\dots \rightarrow Y\rightarrow Y$ instead of passing over the other symbol like $\dots \stackrel{X}{\longrightarrow}Y$ or 
$\dots \stackrel{Y}{\longrightarrow}X$,
and affects states  q1-C1, q1-C2, q1-C3, q1-6, q1-9, q2-1, q4-1, q4-3, q5-0, and q9-4, ten states in all.
Duplicating these states leads to an implementation with $|Q| = 54, |E| =7$ for $n(3)$.

We account for these 10 states in Section~3 by a parameter $\Delta$:\\
$\Delta=0$ for $X\in E, |E|=2k+2$ and
$\Delta=1$ for $X\not\in E, |E|=2k+1$.

\bt{c|c|c|c|c|c|l}
\multicolumn{7}{l}{1. @I.Copy $s$ from II to III, N/2 times, separated by $+$.}\\
q1-1&-- &{\$},+&1,2 &Y   &\mca&We increase $i\leq N/2$ in unary\\
\cline{2-5}
\xb &1,R &2,R   &=,R &=,R &\mca&+,2: already increased\\
    &q1-2&q1-2  &q1-1&q2-1&\mca&--,1: i++. Y: end of number $i=N/2$\\
\cline{1-5}
q1-2&Y   &*\bs X&\mcc&Find right hand side of  I \\
\cline{2-3}
\xb &=,R &=,R   &\mcc&after i++\\
    &q1-4&q1-2  &\mcc&\\
\cline{1-6}
q1-4&--  &\$   &+    &1,2,3&X   &Remove prime from h = --,\$,+$\equiv$1',2',3'\\
\cline{2-6}
\xb &1,R  &2,R  &3,R  &=,R  &=,R &and GOTO respective state q1-Ch\\
    &q1-C1&q1-C2&q1-C3&q1-4 &q1-7&X: end of $s_i$\\
\cline{2-6}
    &[Y] &\mcd&Y: Part of INIT, not q1-4\\
    &X,R &\mcd&\\
    &q3-1&\mcd&\\
\cline{1-3}
q1-C1&Y   &*    &\mcc&Find rhs of III and put symbol 1\\
\cline{2-3}
\xb  &1,L &=,R  &\mcc&\\
     &q1-6&q1-C1&\mcc&\\
\cline{1-3}
q1-C2&Y   &*    &\mcc&... symbol 2\\
\cline{2-3}
\xb  &2,L &=,R  &\mcc&\\
     &q1-6&q1-C2&\mcc&\\
\cline{1-3}
q1-C3&Y   &*    &\mcc&...  symbol 3\\
\cline{2-3}
\xb  &3,L &=,R  &\mcc&\\
     &q1-6&q1-C3&\mcc&\\
\cline{1-3}
q1-6&Y   &*   &\mcc&Find lhs of II $=s_1s_2\dots$\\
\cline{2-3}
\xa &=,R &=,L &\mcc&\\
    &q1-4&q1-6&\mcc&\\
\cline{1-3}
q1-7&Y   &*\bs \$,X&\mcc&Find rhs of III and put\\
\cline{2-3}
\xb &+,L &=,R      &\mcc&interword gap +\\
    &q1-9&q1-7     &\mcc&\\
\cline{1-4}
q1-9&X    &[Y,--] &*\bs \$&\mcb&Find rhs of II\\
\cline{2-4}
\xc &X,L  &X,L  &=,L    &\mcb&Y,--: MERGE from q9-5\\
    &q1-10&q1-10&q1-9   &\mcb&\\
\cline{1-6}
q1-10&1   &2    &3    &Y    &*\bs\$&Mark 1,2,3 as --,\$,+$\equiv$1',2',3'\\
\cline{2-6}
\xa &--,L&\$,L &+,L  &=,L  &=,L   &\\
    &q1-10&q1-10&q1-10&q1-11&q1-10 &\\
\cline{1-6}
q1-11&Y  &*    &\mcc&Find lhs of I\\
\cline{2-3}
\xa &=,R &=,L  &\mcc&\\
    &q1-1&q1-11&\mcc&\\
\cline{1-3}
\et

\bt{c|c|c|c|c|c|l} 
\multicolumn{7}{l}{2. @II. Cut away left triangle in III.}\\
\cline{1-3}
q2-1&Y   &*   &\mcc&Find rhs of III\\
\cline{2-3}
\xb &=,L &=,R &\mcc&\\
    &q2-2&q2-1&\mcc&\\
\cline{1-4}
q2-2&X   &+    &*\bs Y &\mcb&Change last interword gap + to \$\\
\cline{2-4}
\xa &=,R &\$,R &=,L    &\mcb&\\
    &q3-1&q2-3 &q2-2   &\mcb&\\
\cline{1-4}
q2-3&1,2,3&Y   &--   &\mcb&Clear one symbol to --\\
\cline{2-4}
\xb &--,R &=,L &=,R  &\mcb&\\
    &q2-4 &q2-2&q2-3 &\mcb&\\
\cline{1-4}
q2-4&--  &Y   &*\bs +,X&\mcb&Find next unfinished triangle to the right\\
\cline{2-4}
\xb &--,R&=,L &=,R    &\mcb&Y: all tringles to the right finished\\
    &q2-3&q2-2&q2-4   &\mcb&\\
\cline{1-4}
\cline{1-7}
\multicolumn{7}{l}{ }\vspace*{-3mm}\\
\multicolumn{7}{l}{3. @III.Cut away double right triangle in III.}\\
\cline{1-4}
q3-1&{\$}&Y   &*\bs X &\mcb&Find first interword gap \$\\
\cline{2-4}
\xb &+,R &=,L &=,R    &\mcb& from the left. Change to +\\
    &q3-5&q4-1&q3-1   &\mcb&work from here to rhs of III\\
\cline{1-4}
q3-5&Y   &--  &+   &\mcb&Necessary in case the \$$\to$ + in q3-1\\
\cline{2-4}
\xc &=,L &=,L &=,L &\mcb&is directly left of the Y marker\\
    &q4-1&q3-5&q3-2&\mcb&\\
\cline{1-4}
q3-2&1,2,3&*\bs XY\$&\mcc&Skip intermediate --,+ then clear one symbol to --\\
\cline{2-3}
\xa &--,L &=,L      &\mcc&\\
    &q3-3 &q3-2     &\mcc&\\
\cline{1-4}
q3-3&1,2,3&[Y] &[--]&\mcb&Clear another symbol to --\\
\cline{2-4}
\xc &--,L &=,R &=,R  &\mcb&Y,--: MERGE from q9-3\\
    &q3-4 &q9-4&q3-3 &\mcb&(+,X,\$ will not appear)\\
\cline{1-4}
q3-4&--,+&X   &*\bs Y,\$&\mcb&Go left before next word\\
\cline{2-4}
\xa &=,L &=,R &=,L   &\mcb&\\
    &q3-2&q3-1&q3-4  &\mcb&\\
\cline{1-4}
\et

Symmetry between q2 and q3:

$Y\leftrightarrow X, R \leftrightarrow L$,
q2-2$\equiv$q3-1,
q2-3$\equiv$(q3-2,q3-3),
q2-4$\equiv$q3-4.

\bt{c|c|c|c|c|c|l} 
\multicolumn{7}{l}{4. III@. Remove $(i-1)$ patterns in III.}\\
\cline{1-3}
q4-1&Y   &*   &\mcc&Find rhs end of I\\
\cline{2-3}
\xa &=,L &=,L &\mcc&\\
    &q4-2&q4-1&\mcc&\\
\cline{1-5}
q4-2&{\$}&2   &1,3,--,+ &Y\bs X&\mca &Increase $l:=1,\dots,i$ (unary) in I\\
\cline{2-5}
\xa &+,R &3,R &=,L      &=,R   &\mca&If Y, done: $l==i==lmax$ \\
    &q4-3&q4-3&q4-2     &q5-0  &\mca&\\
\cline{1-5}
q4-3&X   &*   &\mcc&Find lhs end of III\\
\cline{2-3}
\xb &=,R &=,R &\mcc&\\
    &q4-4&q4-3&\mcc&\\
\cline{1-3}
q4-4&+   &*\bs X,Y,\$&\mcc&Clear all 1,2,3 to -- until ...\\
\cline{2-3}
&--,L&--,R&\mcc&+: end of pattern\\
\xb &q4-1&q4-4&\mcc&\\
\cline{1-7}
\multicolumn{7}{l}{ }\vspace*{-3mm}\\
\multicolumn{7}{l}{5. III, starts at I@. Check patterns for $s^{(i)} \subset s^{(j)}, j>i$.}\\
\cline{1-3}
q5-0&X   &*   &\mcc&Find lhs end of III\\
\cline{2-3}
\xb &=,R &=,R &\mcc&\\
    &q5-1&q5-0&\mcc&\\
\cline{1-4}
q5-1&1    &2    &3    &\mcb&Take and delete \\
\cline{2-4}
\xb &--,R &--,R &--,R &\mcb&leftmost symbol from $s^{(i)}$\\
    &q5-V1&q5-V2&q5-V3&\mcb&\\
\cline{2-4}
    &-- &+   &Y   &\mcb&\\
\cline{2-4}
    &=,R &+,R &=,L &\mcb&+: this $i$ done\\
    &q5-1&q6-1&q6-2&\mcb&Y: same, $i==N/2$\\
\cline{1-4}
q5-{\bf Vh}&--,+ &Y   &*    &\mcb&Skip rest of $s^{(i)}$ or $s^{(j)}$\\
\cline{2-4}
\xb &=,R  &=,L &=,R  &\mcb&${\bf h} \in\{ 1,..,k\}$ is symbol\\
    &q5-{\bf Kh}&q5-2&q5-{\bf Vh}&\mcb&\\
\cline{1-5}
q5-K1&1   &+    &Y   &*    &\mca&Clear all non-matching in all $s^{(j)}$\\
\cline{2-5}
\xb &\$,R &=,R  &=,L &--,R &\mca&until match 1, set to \$,\\
    &q5-V1&q5-K1&q5-2&q5-K1&\mca&then skip to next word in q5-Vh\\
\cline{1-5}
q5-K2&2   &+    &Y   &*    &\mca&Clear all others until match 2\\
\cline{2-5}
\xb &\$,R &=,R  &=,L &--,R &\mca&\\
    &q5-V2&q5-K2&q5-2&q5-K2&\mca&\\
\cline{1-5}
q5-K3&3   &+    &Y   &*    &\mca&Clear all others until match 3\\
\cline{2-5}
\xb &\$,R &=,R  &=,L &--,R &\mca&\\
    &q5-V3&q5-K3&q5-2&q5-K3&\mca&\\
\cline{1-5}
q5-2&X   &*   &\mcc&Go to lhs of III\\
\cline{2-3}
\xa &=,R &=,L &\mcc&\\
    &q5-1&q5-2&\mcc&\\
\cline{1-3}
\et

\bt{c|c|c|c|c|c|l} 
\multicolumn{7}{l}{6. in III. Remove III to $\varepsilon$Y$^\omega$. Check for match  $s^{(i)}\subset s^{(j)}$.}\\
\cline{1-3}
q6-1&Y   &*   &\mcc&Find rhs end of III\\
\cline{2-3}
\xb &=,L &=,R &\mcc&\\
    &q6-2&q6-1&\mcc&\\
\cline{1-4}
q6-2&\$  &X   & *  &\mcb&\$: Match!\\
\cline{2-4}
\xa &Y,L &=,L &Y,L &\mcb&X: III == $\varepsilon$Y$^\omega$, NO match\\
    &q6-3&q8-0&q6-2&\mcb&*: fill with blanks (Y), NO match\\
\cline{1-4}
q6-3&X   &*   &\mcc&X: III == $\varepsilon$Y$^\omega$, match!\\
\cline{2-3}
\xa &=,L &Y,L &\mcc&*: fill with blanks (Y)\\
    &q7-1&q6-3&\mcc&same as q6-2, but match!\\
\cline{1-3}
\multicolumn{7}{l}{The {\$} is the last symbol from $s^{(i)}$ matching some symbol from $s^{(j)}$,}\\
\multicolumn{7}{l}{implying that all previous symbols from  $s^{(i)}$ were also matched.}\\
\cline{1-7}
\multicolumn{7}{l}{ }\vspace*{-3mm}\\
\multicolumn{7}{l}{7. II@.  $s++$, HALT if $s=3^N$, otherwise $i := 0; j := 0$.}\\
\cline{1-5}
q7-1&-- &\$   &+ &Y &\mca&$s++$, 1'$\to$2'$\to$3'$\to$(1'+carry)\\
\cline{2-5}
\xa &\$,L &+,L &--,L &=,R &\mca&HALT is the actual HALT after reaching $n_3$!\\
    &q7-3&q7-3&q7-1 &HALT &\mca&Replace \$ by(--,L,q7-1) for $|B|=2$\\
\cline{1-5}
q7-3&Y   &*   &\mcc&Find lhs of II\\
\cline{2-3}
\xa &=,L &=,L &\mcc&\\
    &q7-4&q7-3&\mcc&\\
\cline{1-3}
q7-4&Y   &*\bs --&\mcc&i := 0; j := 0. Goto lhs of I\\
\cline{2-3}
\xa &=,R &--,L  &\mcc&\\
    &q1-1&q7-4  &\mcc&\\
\cline{1-3}
\cline{1-7}
\multicolumn{7}{l}{ }\vspace*{-3mm}\\
\multicolumn{7}{l}{8. @III. $i:=0, lmax++$. IF $lmax < N/2$ GOTO 1 ELSE GOTO 9.}\\
\cline{1-3}
q8-0&Y   &*   &\mcc&Find rhs of I\\
\cline{2-3}
\xc &=,L &=,L &\mcc&\\
    &q8-1&q8-0&\mcc&\\
\cline{1-6}
q8-1&--,\$,+&1   &2   &3   &Y\bs X&$i := 0$, $l$ unchanged\\
\cline{2-6}
\xa &=,L   &--,L&\$,L&+,L &=,R   &\\
    &q8-1  &q8-1&q8-1&q8-1&q8-2  &\\
\cline{1-6}
q8-2&-- & Y   &*   &\mcb&$-\to\$: lmax$++ (support for $l$), GOTO 1\\
\cline{2-4}
\xb &\$,L &=,L &=,R &\mcb&Y: Since $lmax$==N/2, GOTO 9\\
    &q1-11&q9-1&q8-2&\mcb&\\
\cline{1-4}
\et

\bt{c|c|c|c|c|c|l} 
\multicolumn{7}{l}{9. I@.  $(N/2==imax)++,$ sXY$\to$s--XY, $i := 0, l := 0$.}\\
\cline{1-3}
q9-1&Y   &*\bs X&\mcc\hspace*{3cm}&Find lhs of I. Extend I and thus\\
\cline{2-3}
    &--,R&--,L  &\mcc& N/2 by a new --. $i:= 0, l := 0$\\
    &q3-3&q9-1  &\mcc&Y: end, goes to q9-3 merged into q3-3\\
\cline{1-3}
{[q9-3]}&Y   &-- &\mcc&Find rhs of I\\
\cline{2-3}
       &=,R &=,R &\mcc&is merged into q3-3 \\
        &q9-4&q9-3&\mcc&\\
\cline{1-3}
q9-4&Y   &*&\mcc&Find rhs of III, extend by XY$\to --$\\
\cline{2-3}
    &--,R&--,R  &\mcc&Y: end, goes to q9-5 merged into q1-9\\
    &q1-9&q9-4  &\mcc&With q1-9, extend by  $-\equiv 1'$\\
\cline{1-3}
{[q9-5]}&Y    &\mcd&Set YY$^\omega$ to XY$^\omega$ (III==$\varepsilon$)\\   
\cline{2-2}
        &X,L  &\mcd &is merged into q1-9 \\
        &q1-10&\mcd&\\
\cline{1-2}
\et

\section{Varying $|E|$ down to $|E'|=2$}

We now have a $|Q\times E| = 44\cdot 8$ implementation (as well as a $54\cdot 7$ one).

Next, to bring the symbol count down to $|E'| = 2$ or 3, we may map up to 8 symbols to triples of bits and up to 9 symbols to pairs of ternary ``trits''.

The general case has the following relevant parameters:
The new alpha\-bet size $b=|E'|$, usually 2 or 3,
the length or the size of the $l$-tuple of $b$-ary digits simulating one original symbol from $E$,
$l = \lceil\log_b|E|\rceil$, and
$k$ is from our problem $n(k)$, usually $k=3$ or 4.

Every state is simulated in up to 4 sweeps through the $l$-tuple:

Sweep-0: If a state is entered alternately from both sides, Sweep-0 uses $l-1$ substates to bring the head from the ``other'' end to the normal one.
If a state is always entered from the same side, Sweep-0 is skipped.

Sweep-1 moves through the $l$-tuple to obtain the current symbol $e$, with $1,b,b^2,\dots,b^{l-1}$ states in the successive positions.

For ``Scan states'' where we search for 1 symbol and all others are combined in the wildcard case '*', the count is upperbounded by $2l-1$ substates, since in each  position (after the first) we only have to  distinguish ``scan symbol still possible'' vs.~``is some other symbol''.

Sweep-2 moves back to replace $e$ by $E^+(q,e)\neq e$ and or just to reach the other end to leave the $l$-tuple.
It uses $l-1$ substates per case.

Sweep-3 moves forth again, in $l-1$ substates per case, if we had to replace $e$ by $E^+(q,e)$, but leave opposite to the entry side.

Our $35+3k$ states have between $v=2$ and $v=6$ different cases.
$\#_{Sweep-0}=3+\Delta$  of them (q1-9,q3-3,q3-5), q1-9 is duplicated for $\Delta=1$, are entered from both sides.

$\#_{Scan>} =6+\Delta\cdot 3$ states (q1-2,q4-1,q4-3,q5-0,q7-3,q8-0) are scan states that are left opposite of the entry side and do not change the symbol.
Hence only Sweep-1 with $(1,2,2,2,\dots)$, {\it i.e.} $2l-1$,  substates is needed.

\begin{figure}
\centering
  \bt{l|c|cccc}
  state&cases&${>}$&${<}$&${<}$&${>}$\\
  &  &${=}$&${=}$&${\neq}$&${\neq}$\\
  \cline{1-6}
Scan$>$&$6\cdot 2$&12&-&-&-\\  
Scan$<$&(6+k) $\cdot$ 2&6+k&6+k&-&-\\  
q1-1&4&2&-&-&2\\
q1-4&3+k&2&-&-&1+k\\
q1-Ch&2k&k&-&k&-\\
q1-9&2&-&1&1&-\\
q1-10&2+k&2&-&-&k\\
q2-2&3&1&1&1&-\\
q2-3&3&1&1&-&1\\
q2-4&3&1&1&-&1\\
q3-1&3&1&1&-&1\\
q3-5&3&-&3&-&-\\
q3-2&2&1&-&-&1\\
q3-3&3&2&-&1&-\\
q3-4&3&2&1&-&-\\
\cline{1-6}  
\et
\bt{|l|c|cccc}
  state&cases&${>}$&${<}$&${<}$&${>}$\\
  &  &${=}$&${=}$&${\neq}$&${\neq}$\\
\cline{1-6}
q4-2&4&1&1&2&-\\
q4-4&2&-&-&1&1\\
q5-1&k+3&1&1&-&1+k\\
q5-Vh&3k&2k&k&-&-\\
q5-Kh&4k&k&k&-&2k\\
q6-2&3&1&1&1&-\\
q6-3&2&1&-&-&1\\
q7-1&1+k&1&-&-&k\\
q7-4&2&-&1&-&1\\
q8-1&5&1&1&-&3\\
q8-2&3&1&1&1&-\\
q9-1&2&-&-&1&1\\
q9-4&2&-&-&-&2\\
\\
\\
\cline{1-6}
  \et
  
  \bt{cccccccl}
  \multicolumn{6}{l}{ }\vspace*{-3mm}\\
  &states&cases&$>=$&$<=$&$<\neq$&$>\neq$\\
  $\sum $&35+3k&63+13k&22+4k & 15+2k & 9+k & 17+6k&$\Delta = 0$\\
  $\sum $& 7+ k& 14+2k & 8+k &     3 & 1+k &  2   &$\Delta = 1$: Extra cases\\
  \et
  \caption{Details of states \label{fig:substates}}
\end {figure}

Another $\#_{Scan<} =6+k+\Delta(2+k)$ states (q1-Ch,q1-6,q1-7,q1-11,q2-1,q5-2,q6-1) are scan states left at the entry end in case of a match, with or without changing $e$ to $E^+$.
Here we have Sweep-1 as before, and $l-1$ substates for Sweep-2 of the match case.

Then there are $\#_{Scan\neq}=5+\Delta$ states (q4-4,q6-3,q7-4,q9-1,q9-4) that scan in Sweep-1, but afterwards change $e$ to $E^+(q,e)\neq e$.
Here we have $2l-1$ substates for Sweep-1, but Sweep-2 and -3 are counted as in the general case.

The further $\#_{Sweep1} = 18+2k+\Delta\cdot 1$ states require $1+b+\dots+b^{l-1}=(b^l-1)/(b-1)$ substates for Sweep-1.
They and the  $\#_{Scan\neq}$ states are given in Table~\ref{fig:substates}.

\begin{figure}
\centering
  \bt{cccc|cccc}
  \multicolumn{4}{l|}{$|E|=...$}&
  \multicolumn{3}{c}{Meaning in segment}\\
  8+  & 2 & 3& 4& I&II&III\\
\cline{1-7}  
  Y  &111&22&33&marker&marker&marker& also Blank\\
  X  &011&21&22& --&marker&marker\\
  1  &100&10&10&$i=1,l=-$&1&1\\
  2  &101&11&11&$i=1,l=0$&2&2\\
  3  &110&12&12&$i=1,l=1$&3&3\\
  -- &000&00&00&$i=0,l=-$&1'&nil\\
  \$ &001&01&01&$i=0,l=0$&2'&match\\
  +  &010&02&02&$i=0,l=1$&3'&end of word\\
\cline{1-7}  
  4  &--&--&13&--&4&4\\
  4' &--&--&03&--&4'&---\\
\et
\caption{Synopsis of symbols over various alphabets \label{fig:Symb}}
\end {figure}

The total number of substates according to this upper bound is then
$\ds {\tt [Sweep-0]}\quad (l-1)\cdot\#_{Sweep-0}+$\\
$\ds{\tt [Sweep-1]}\quad(2l-1)\cdot (\#_{Scan>}+\#_{Scan\neq}) + (3l-2)\cdot \#_{Scan<}$\\
$\ds\hspace*{2 cm} +(1+b+b^2+\dots +b^{l-1})\cdot\#_{Sweep-1}+$\\
$\ds{\tt [Sweep-2/3]}\quad(l-1)\cdot (0\cdot v_{>=}+ 1\cdot(v_{<=}+v_{<\neq})+2\cdot v_{>\neq})$

Note the dual use of the factor $k$ in Figure~\ref{fig:substates} :
In q1-4,q1-10,15-1, we have to distinguish $k$ symbols in $k$ different cases.
In q1-Ch,q5-Vh,q5-Kh, there are $k$ different states, each one with a constant amount of cases.

\begin{figure}[b]
\centering
\bt{lll}
Sweep&$\Delta=0, |E|=2k+2$&$\Delta=1, |E|=2k+1$\\
\cline{1-3}
Sweep-0&$3\cdot (l-1)$&$4\cdot (l-1)$\\
Sweep-1&$(17+k)\cdot(2l-1)$&$(23+2k)\cdot(2l-1)$\\
       &$+(18+2k)\cdot (1+b+\dots+b^{l-1})$&$+(19+2k)\cdot (1+b+\dots+b^{l-1})$\\
Sweep-2/3&$(6+k+58+15k)\cdot (l-1)$&$(8+2k+66+16k)\cdot (l-1)$\\
\et
\caption{Substate counts per sweep \label{fig:SweepC}}
\end{figure}

For $b=2,l=3$ we have\\
$\ds2\cdot 3+5\cdot(6+6)+7\cdot(6+k)
+(1+2+4)\cdot(17+2k)
+2\cdot(2k+15+k+9+2\cdot(6k+17))$
$=343+51k$ substates.
For $k=3$ that is 496 substates, to be compared with the actual 276 states for the ``hand-wired'' version.

For $b=3,l=2$ we have\\
$\ds 1\cdot 3+3\cdot(6+6)+4\cdot(6+k)
+(1+3+9)\cdot(17+2k)
+1\cdot(2k+15+k+9+2\cdot(6k+17))$\\
$=351+45k$ substates, which is 486 for $k=3$,  while the ``hand-wired'' version for pairs of trits $(l=2, b=3)$ needs only 155 states.

For $b=2$ and $l=4$ we have\\
$\ds 3\cdot 3+7\cdot(6+6)+10\cdot(6+k)
+(1+2+4+8)\cdot(17+2k)
+3\cdot(2k+15+k+9+2\cdot(6k+17))$
$=582+85k$.
With $k=4$, we get a $(922,2)$ implementation for $n(4)$.

\begin{theorem}\label{thm:Eb}
  $(i)$ Given any TM with $|Q|=n, |E|=m$, we obtain another TM with $|E'|=b$, $l:=\lceil \log_b|E|\rceil$  and a state count of at most  $|Q'| \leq $
  $$n\cdot [(l-1)+(1+b+b^2+\dots+b^{l-1})+|E|\cdot 2(l-1)] < 2(m+1)n\lceil\log_b(m)\rceil.$$

  $(ii)$ With $n,m,b,l$ as before and the numbers $\#_{Sweep-0},\#_{Scan},\#_{>=},\#_{>\neq},$ $\#_{<=},\#_{<\neq}$ given as in the text, we obtain the sharper upper bound\\
  $\ds |Q'|\leq\#_{Sweep-0}\cdot (l-1)+\#_{scan}\cdot(2l-1) + (n-\#_{scan})(1+b+\dots +b^{l-1})$\\
  $\ds+(\#_{<=}+\#_{<\neq}+2\#_{>\neq})\cdot(l-1)$\\
  $\ds< n\left(\frac{b^l-1}{b-1}+l-1\right)+\#_{Sweep-0}\cdot (l-1)-\#_{Scan}
\left(\frac{b^l-1}{b-1}-2l+1\right) $.
\end{theorem}

\begin{proof}
  $(i)$ We use $l-1<\log_b(m), 1+b+\dots+b^{l-1}=\frac{b^l-1}{b-1} < 2m$.

  $(ii)$ We use $(\#_{<=}+\#_{<\neq}+{\bf 1}\cdot\#_{>\neq})\leq n$.
\end{proof}

The bounds in Theorem~\ref{thm:Eb} are independent of the mapping $E\to E'$, which may improve the numbers.

\section{Varying $|Q|$ down to $|Q'|= 2$}

\subsection{$|Q'|=3$}

Reduction in states down to 3 is achieved via simulation.

We use only three states, qX (expansion), qL (go left), and qR (go right), so $|Q'|=3$.
The new symbol set is $E' = Q\times \{X,L,R\}\times E\ni (q,d,e)$,
where we assemble the new state by transferring it bit-by-bit via qL or qR, successively yielding the $q$ part of $e'=(q,e,d)$.

The generalized direction $d$ tells, whether the state still has to be expand\-ed, $d=X$, or the direction is  $d\in\{L,R\}$ as in the original TM.
The $e$ part is the original symbol.

The part $e$ of the current symbol $e$ is replaced by $e' = Q^+(q,e)$ at expansion.
The direction is $d=$ X during the assembly of $q$ and then $d = D^+(q,e)$.

\begin{figure}[t]
\centering
  $\ba{cccccccl}
   &q&e& & Q^+ & E^+ & D^+\\
\cline{1-7}
  (1)&qX& [q,d,e] &\to&  qd& [q-1,d,e]&d&q\in\nn\\
  (2)&qL& [q,X,e] &\to&  qX& [q+1,X,e]&R\\
  (3)&qR& [q,X,e] &\to&  qX& [q+1,X,e]&L\\
  (4)&qX& [0,d,e] &\to&  qX& [0  ,X,e]&d&d\in{L,R}\\
  (5)&qX& [q,X,e] &\to&  q\tilde d& [\tilde q-1,\tilde d,\tilde e]&\tilde d\\
  \ea$
\caption{Transitions for $|Q'|=3$ \label{fig:Trans3}}.
\end {figure}

We first use transition (1) from Figure~\ref{fig:Trans3}
with $d\in \{L,R\}$ that is we move to the left with $qL$ for $d=L$ or to the right with $qR$ for $d=R$, respectively.
Then we return from the neighbour cell, whose $q$ is incremented,  using transition (2) or (3).
We repeat until $q=0$ with $qX$, and then release, $d:=X$, the current cell via (4), thus finishing this current cell.

We move over to the neighbor as new current cell by (4), staying in state $qX$.
Here, we expand the symbol $(q,e,X)$ by transition (5):
We first calculate the new values $\tilde q := Q^+(q,e), \tilde e := E^+(q,e)$, and $\tilde d := D^+(q,e)$ with the functions of the original ma\-chine.
The intermediate result would be $(qX, [\tilde q,\tilde d,\tilde e],\tilde d)$. However, we immediately include the first move, decre\-ment\-ing $\tilde q$, and obtain the overall value, starting the bitwise transfer of $\tilde q\in\nn$ in the new direction $\tilde q$.

Two remarks:
The tape's blank symbol is $(0,X,-)$ where $-$ is the original blank from $E$, and upon halting the original machine, $Q^+$=HALT, also the simulator actually halts.

\begin{theorem}\label{thm:Q3}
  Given any TM with $|Q|=n, |E|=m$, we obtain another TM with $|Q'| = 3$ and at most
  $$|E'|\leq 3(n+1)\cdot m$$
  states.
\end{theorem}

\begin{proof}
  From Figure~\ref{fig:Trans3}, we have a possible
  $Q':= \{qX,qL,qR\}$ and\\
  $E' := (Q\cup\{0\})\times \{L,R,X\}\times E$ with the given sizes.
\end{proof}

\subsection{$|Q'|=2b+1, b\geq 2$}

With $|Q'|=3$, we used a one-letter  alphabet for the ``information transfer'' that is $\log(1)=0$ bits of information to be transferred in each move.

\begin{figure}
\centering
  $\ba{ccc|cccl}
  &q&e&Q^+&E^+&D^+&\\
\cline{1-6}
  (1)&qD_{\overline{q_a}}&[q_{a-1}\dots q_1q_0,-,e]&qX&[q_aq_{a-1}\dots q_1q_0,-,e]&\overline D\\
(2)&qX&[q_{n-1}q_{n-2}\dots q_a,D,e]&qD_{q_a}&[q_{n-1}q_{n-2}\dots q_{a+1},D,e]&D\\
(3)&qX&[q_{n-1}q_{n-2}\dots q_1q_1,-,e]&qD'_{q'_{n-1}}&[q'_{n-2}\dots q'_{1}q'_0,D',e']&D'\\
\multicolumn{6}{l}{\mbox{\rm with QED}^+(q_n-1\dots q_0,e) = (q'_{n-1}\dots q'_0,e',D')\mbox{\rm\ in the original TM}}\\
  \ea$
\caption{Transitions for $|Q|=2b+1$ \label{fig:Transb}}
\end {figure}

As Chaitin \cite{Chaitin} points out, the information about the {\it length} of the trans\-mission, or the {\it end} of the transfer is a necessary and important piece of information.
Here it was the {\it only} in\-for\-ma\-tion.

We can, however, use larger alphabets with $b$ letters and 
$Q' =\{qX,qL_0,$ $qL_1,\dots,qL_{b-1},qR_0,\dots qR_{b-1}\}$,  thus $|Q'| = 2b+1$.


Our symbol set then is:

\begin{definition} Symbol set $E'$  for $|Q'|=2b+1$\label{defSySet}
  
Let $E' :=$
$\ds \{-, (q'_{l-1}q_{l-2}\dots q_1)_b, \dots(q'_{l-1}q_{l-2})_b, (q'_{l-1})_b\}  \}\times E\times \{L,R\}$\hspace*{1 cm}\\
$\ds\dot{\cup}\ \ 
(\{-\}\dot{\cup} \{(q'_{l-1}q_{l-2}\dots q_1q_0)_b\}
\dot{\cup}[b]^{l-1}\dot{\cup}[b]^{l-2}\dot{\cup}\dots \dot{\cup}[b])\times \{-\}\times E$\\
where $l := \lceil\log_b|Q|\rceil$, $(q'_{l-1}q_{l-2}\dots q_0)_b\leq |Q|$ and $[b] := \{0,1,\dots,b-1\}$.

The  numbers in the first part describe $|Q|/b, |Q|/b^2,\dots |Q|/b^{l-1}$ prefixes, the numbers in the second part $|Q|, b^{l-1},\dots,b$ suffixes of the elements from $Q\equiv \{1,\dots,|Q|\}$.

We have $E'\leq (1+\lceil n/b\rceil +\lceil n/b^2\rceil +\dots+\lceil n/b^{l-1}\rceil)\cdot 2m+nm+\frac{b^l-1}{b-1}\cdot m $\\
$<  \left[n\cdot \frac{b+1}{b-1}+2(l-1)+ \frac{b^l-1}{b-1}  \right]\cdot m$, using
    $1+\frac{2}{b-1}\cdot\frac{b^{l-1}-1}{b^{l-1}} < \frac{b+1}{b-1}$.
\end{definition}

The state $X$ corresponds to being in the current tape cell, the states $L_i, R_i, 0\leq i\leq b-1$ are used in the adjacent cell to the left or right, respectively.

Let the current state and symbol be
$qX$ and $[(q_fq_{f-1}\dots q_1q_0), L, e], f\geq 1$.
Then $D^+=L$ from the last component of the symbol.
We cut off one $b$-ary digit, $q_f$, which goes into the nextstate $Q^+=qL_{q_f}$, and obtain $E^+=(q_{f-1}\dots q_1q_0, e, L)$, transition (2) in Figure~\ref{fig:Transb}.

Let now the current state be $qL_i$ after moving to the left.
Let the symbol there be $[(q_{l-1}q_{l-2}\dots q_{f+1}), -,e]$.
Then $Q^+=qX$, $D^+=R$, the opposite direction from $q=L_i$.
Also, let $E^+=[(q_{l-1}q_{l-2}\dots q_{f+1}i),-,e]$, appending the $b$-ary digit transferred as index $i$ to the {\it right} of the current first component, there becoming $q_f$ as in transition (1) of Figure~\ref{fig:Transb}.

{\bf Example} (see Figure~\ref{fig:Ex3})

\nopagebreak

We run two transitions
$QED^+(7,e_1) = (15,e_2,L)$ and 
$QED^+(15,e_3) = (4,e_5,R)$ of the original machine, where $7,15,4\in Q$ and $e_1,e_2,e_3,e_5\in E$.

\begin{figure}[t]
\centering
  \bt{cclcclccl}
  &Pos&q&e&$\to$&$Q^+$&$E^+$&$D^+$&\\
\cline{2-8}
  &102:&qX&[021,$-$,$e_1$]&$\to$&qL$_1$&[20,L,$e_2$]&L&with $(120)_3=15=Q^+$\\
  &101:&qL$_1$&[--,$-$,$e_3$]&$\to$&qX&[1,$-$,$e_3$]&R\\
  &102:&qX&[20,L,$e_2$]&$\to$&qL$_2$&[0,L,$e_2$]&L&\\
  &101:&qL$_2$&[1,$-$,$e_3$]&$\to$&qX&[12,$-$,$e_3$]&R\\
  &102:&qX&[0,L,$e_2$]&$\to$&qL$_0$&[--,L,$e_2$]&L&\\
  &101:&qL$_0$&[12,$-$,$e_3$]&$\to$&qX&[120,$-$,$e_3$]&R&$q=15=(120)_3$ has\\
  &&&&&&&&arrived completely\\
  &102:&qX&[--,L,$e_2$]&$\to$&qX&[--,$-$,$e_2$]&L&Clean-up, note $Q^+=X$\\
(*) &101:&qX&[120,$-$,$e_3$]&$\to$&qR$_1$&[01,R,$e_5$]&R&QED$^+(15,e_3)$\\
  &102:&qR$_1$&[--,$-$,$e_2$]&$\to$&qX&[1,$-$,$e_2$]&L&\\
  &101:&qX&[01,R,$e_5$]&$\to$&qR$_1$&[0,R,$e_5$]&R&\\
  &102:&qR$_1$&[1,$e_2$,--]&$\to$&qX&[11,$-$,$e_2$]&L&\\
  &101:&qX&[0,R,$e_5$]&$\to$&qR$_0$&[--,R,$e_5$]&R&\\
  &102:&qR$_0$&[11,$-$,$e_2$]&$\to$&qX&[011,$-$,$e_2$]&L&\\
  &101:&qX&[--,R,$e_5$]&$\to$&qX&[--,$-$,$e_5$]&R&\\
(**) &102:&qX&[011,$-$,$e_2$]&$\to$&qD$_i$&[q',D,$e_x$]&D&QED$^+(4,e_2)$, D $\in\{L,R\}$\\
    \et
\caption{Example for $b=3$, $|Q|=7$ \label{fig:Ex3}}
\end {figure}

We assume to be {\it e.g.} in position 102, with position 101 holding some symbol $e_3$ and position 103 holding $e_4$, hence the simulator has

101: $(-,-,e_3)$,\ \
102: $(021,-,e_1)$,\ \ 
103: $(-,-,e_4)$
on its tape where $(021)_3=7$ is the current state.

In $(*)$ of Figure~\ref{fig:Ex3} we use transition (3) of  Figure~\ref{fig:Transb} and expand according to the original transition $(15,e_3)\to (4,e_5,R)$.

The new state $(011)_3$ is divided into (01) in the symbol and the trailing 1 as index to $q=R_1$.

The same happens in $(**)$, where now QED$^+(4,e_2) = (q'|i,e_x,d)$ defines the new parts $(3\cdot q'+i)\in Q, e_x\in E, d\in\{L,R\}$.
{\it E.g.} for QED$^+(4,e_2) = (8,e_6,L)$ with $8=(022)_3$, we have $q'=02, i=2, e_x=e_6, D=L$ and thus $(L_2,[02,L,e_6],L)$ as r.h.s.


\begin{theorem}\label{thm:Q2b}
  Given any TM with $|Q|=n, |E|=m$, there is another TM with $|Q'| = 2b+1$ states and at most
  $$|E'|\leq
\left(n\cdot \frac{b+1}{b-1}+2(l-1)+ \frac{b^l-1}{b-1}  \right)\cdot m$$
  symbols, where $l := \lceil\log_b(n)\rceil$.
\end{theorem}

\begin{proof}
  Set $Q'=\{X,L_1,\dots, L_{b-1},R_1,\dots, R_{b-1}\}$, $E'$ as in Def.~\ref{defSySet}, and QED$^+$ as in Figure~\ref{fig:Transb}.
\end{proof}

\subsection{$|Q'|=2$, One Initial Non-Blank Symbol}

We use $Q'=\{L,R\}$,
$E':= \{0,1,2,\dots,|Q|\}\times \{-,L_{new},L_{old},$ $R_{new},R_{old}\}\times E$,
and the transitions from Figure~\ref{fig:Q2}

The meaning of state L (respectively R) here is {\it being in} the left (right) of the two active cells.
If we have a transfer to the right (for $D^+=R$) the $L$ state in $R_{old}$ decrements the $q$ part of its symbol (transition (3) with $X=L, \ox=R$), while the R state in $R_{new}$ increments it, a transition (2), with $X=R, \ox=L$.

We start in state $L$ on the non-blank symbol $(1,R_{new},e_0)$, the 1 denoting the start state, whereas the rest of the tape is blanked out with $(0,-,e_0)$ or any other initial contents $(0,-,e_n)$.
We immediately execute a QED$^+$, transition (5), according to $(q_1,e_0)$ of the original TM.

\begin{figure}
  \centering
  $\ba{ccccccl}
&q&e&\to&Q^+&E^+&D^+\\
\cline{1-7}
(1)&X,&[0,-,e]       &\to&\ox,&[1,X_{new},e],&\ox\\
(2)&X,&[n,X_{new},e]  &\to&\ox,&[n+1,X_{new},e],&\ox\\
(3)&X,&[n,\ox_{old},e]&\to&\ox,&[n-1,\ox_{old},e],&\ox\\
(4)&X,&[0,\ox_{old},e]&\to&X,&[0,-,e],&\ox\\
(5)&X,&[n,\ox_{new},e]&\to&X',&[n'-1,X'_{old},e],&X'\ (*)\\
\ea$

$(*)$ with $(n,e) \mapsto (n',e',X')$ in the original machine
  
\caption{Transitions for $X\in Q' = \{L,R\}$, $|Q'|=2$ }.\label{fig:Q2}
\end {figure}

\subsection{$|Q'|=2$, Starting on Empty Tape $^\omega 0^\omega$\\
  Introducing A Third Option via ``Overflow Error''}

Apparently, when seeing a blank symbol $[0,-,e_0]$ (and the {\tt BB} rules demand {\it only} blanks initially), we have to distinguish 3 situations:

\noindent-- we are to the left of the current position, activate this cell, increase the $q$ counter from the blank value 0 to 1 by a transition of type (1) in Figure~\ref{fig:Q2} with $X=L$ and return to the right\\
-- we are to the right of the current position and behave symmetrically\\
-- at start, we have to convert that blank into $(q_1,-,0)$ to start computation, since there is no state yet on the tape.

When working with only $|Q'|=2$ states and initially only 1 symbol, the blank, we, also apparently, can {\it not} distinguish 3 cases.
If, however, we are allowed a single non-blank tape cell, we can do just fine, see Figure~\ref{fig:Q2}.

As we have seen, three states or two states plus one non-blank are sufficient to get the simulation started.
Our task, however, is starting with a blank tape and $Q=\{L,R\}$.
What does really happen then:\\
\hspace*{-2mm}\bt{l}
$(L,[0,-,e_0]) \to (R,[1,L_{new},e_0],R)$\\
$(L,[1,L_{new},e_0]) \to (R,[2,L_{new},e_0],R)$\\ 
$(L,[2,L_{new},e_0]) \to (R,[3,L_{new},e_0],R)$ 
\et
\hspace*{-2mm}\bt{l}
$(R,[0,-,e_0]) \to (L,[1,R_{new},e_0],L)$\\
$(R,[1,R_{new},e_0]) \to (L,[2,R_{new},e_0],L)$\\
$(R,[2,R_{new},e_0]) \to (L,[3,R_{new},e_0],L)$
\et
\\... and so on, {\it ad infinitum}.

Actually, there is no symbol [$q$,...] in $E'$ for $q>|Q|$ and this yields the third option:

Let $X:= D^+(q_1,0)$ be the first move of the original TM.
Then we start in X and set the equivalences:
$X,[|Q|+1,X_{new},0] :\equiv X,[q_1,\ox_{new},0]$ (bootstrap start) and 
$\ox,[|Q|+1,\ox_{new},0] :\equiv \ox,[0,-,0]$ (blank).

That gives us a 2-state empty initial tape simulation for any other TM with empty initial tape.

{\bf Example} (see Figure~\ref{fig:Q2empty})

Let the original machine start in position 1 on empty tape and execute
$(q_1,0)\to (q_2,1,R)$\ \
$ (q_2,0)\to (q_3,2,L)$\ \
$ (q_3,1)\to (q_4,3,L)$\ \
$ (q_4,0)\to (\mbox{HALT})$.

We start with an all-blank $(0,-,e_0)$ tape in state $R = D^+(q_1,0)$ from the original machine.
We assume $|Q|+1=5$.
See Figures~\ref{fig:Q2empty}/\ref{fig:Q2empty2} for details.

\begin{theorem}\label{thm:Q2}
  $(i)$ A TM with $|Q|=n, |E|=m$ and empty initial tape can be simulated by another TM with $n'=2$, $ Q':=\{L,R\}$ and $m'= 5m(n+2)$,
  $E':= \{0,1,2,\dots,|Q|,|Q|+1\}\times \{-,L_{new},L_{old},$ $R_{new},R_{old}\}\times E$.

\begin{figure}[h!]
  \centering
  $\ba{ccccccccl}
\mbox{\rm Tr.}&\mbox{\rm Pos.}&q&e&&Q^+&E^+&D^+\\
(1)&1:&R&[0,-,0]) &\to& L&[1,R_{new},0]&L&\\
(1)&0:&L&[0,-,0]) &\to& R&[1,L_{new},0]&R&\\
(1)&1:&R&[1,R_{new},0]) &\to& L&[2,R_{new},0]&L&\\
(1)&0:&L&[1,L_{new},0]) &\to& R&[2,L_{new},0]&R&\\
\multicolumn{4}{l}{\mbox{\rm ... (until) ...}}\\
  &1:&R&[5,R_{new},0])&\\
(5)&\equiv&R&[1,L_{new},0] &\to& L&[2-1,L_{old},1]&L&\\
&0:&L&[5,L_{new},0]) &&\\
(1)&\equiv&L&[0,-,0] &\to& R&[1,L_{new},0]&R&\\
(3)&1:&R&[1,L_{old},1]) &\to& L& [0,L_{old},1]&L&\\
(2)&0:&L&[1,L_{new},0] &\to& R&[2,L_{new},0]&R&\\
(4)&1:&R&[0,L_{old},1]) &\to& {\bf R}&[0,-,1]&L&\\
(5)&0:&R&[2,L_{new},0]) &\to& R&[3-1,R_{old},2)&R&, QED^+\\
(1)&1:&R&[0,-,1]) &\to& L&[1,R_{new},1]&L&\\
(3)&0:&L&[2,R_{old},2]) &\to& R&[1,R_{old},2]&R&\\
(2)&1:&R&[1,R_{new},1]) &\to& L&[2,R_{new},1]&L&\\
(3)&0:&L&[1,R_{old},2]) &\to& R&[0,R_{old},2]&R&\\
(2)&1:&R&[2,R_{new},1]) &\to& L&[3,R_{new},1]&L&\\
(4)&0:&L&[0,R_{old},2]) &\to& {\bf L}&[0,-,2]&R&\\
(5)&1:&L&[3,R_{new},1]) &\to& R&[4-1,R_{old},3)&R&, QED^+\\
(1)&2:&R&[0,-,0]) &\to& L&[1,R_{new},0]&L&\\
(3)&1:&L&[3,R_{old},3]) &\to& R&[2,R_{old},3]&R&\\
\ea$
\caption{Two states, empty tape (part~1)\label{fig:Q2empty}}
\end{figure}


  $(ii)$ If the simulating TM does not require an empty initial tape, we have $|E'| = 5m(n+1)$, omitting symbols $[|Q|+1,...,...]$.
\end{theorem}

\begin{proof}
  The simulating TM is given by Figure~\ref{fig:Q2}.
\end{proof}

\begin{figure}
  \centering
  $\ba{ccccccccl}
\mbox{\rm Tr.}&\mbox{\rm Pos.}&q&e&&Q^+&E^+&D^+\\
(2)&2:&R&[1,R_{new},0]) &\to& L&[2,R_{new},0]&L&\\
(3)&1:&L&[2,R_{old},3]) &\to& R&[1,R_{old},3]&R&\\
(2)&2:&R&[2,R_{new},0]) &\to& L&[3,R_{new},0]&L&\\
(3)&1:&L&[1,R_{old},3]) &\to& R&[0,R_{old},3]&R&\\
(2)&2:&R&[3,R_{new},0]) &\to& L&[4,R_{new},0]&L&\\
(4)&1:&L&[0,R_{old},3]) &\to& {\bf L}&[0,-,0]&R&\\
(5)&2:&L&[4,R_{new},0]) &\to& \multicolumn{2}{l}{\mbox{HALT}}&&, QED^+
\ea$
\caption{Two states, empty tape  (part~2)\label{fig:Q2empty2}}.
\end{figure}
\newpage
\section{Larger symbol alphabets: $n(k)$}

\subsection{The number $n(4)$}

For an alphabet with 4 symbols, Friedman gives the ``remarkable'' lower bound:

\begin{theorem} {\rm (Friedman\cite[Theorem 8.4]{F2})}
  
  $$n(4) > A^{(A(187196))}(1),$$
  
  where $A(k) = A(k,k)$ and $A^{(n)}(1) = A(A^{(n-1)}(1)), A^{(1)}\equiv A$.
\end{theorem}

\begin{proof}
  The statement appears in \cite[p.~7]{F2} as a Theorem, no proof given there.
\end{proof}

We have $A(1)=2, A(2)=4, A(3) = 2^{2^2} = 16$, $A(4) = A^{(3)}(1)$ (see below)
... up to $A(187196)$  to generate the exponent.

Then we do $A(187196)$ recursions, starting with $A^{(1)}(1) = 2, A^{(2)}(1) = 4,$\\
$$A^{(3)}(1) =2^{2^{\iddots^{2^2}}},
A^{(4)}(1) =A\big(2^{2^{\iddots^{2^2}}}, 2^{2^{\iddots^{2^2}}}\big),$$
each with 65536 2s stacked,...

\noindent to yield the lower bound for $n(4)$.
This number is indeed, to quote Friedman 
{\it ``\ a whole 'nother kettle of fish''} \cite[p.~7]{F2}.

\subsection{TMs for $n(k)$}

How do our TMs change?

We need two more symbols, the 4 and a 4' in part II.
Also, three new states q1-C4, q5-K4, q5-V4 deal with the additional symbol.
Hence, the $(44,8)$ implementation for $n_3$ yields a $(47,10)$ implementation for $n(4)$, and in general there is a $(35+3k,2+2k)$ implementation for computing $n(k), k\geq 3\in\nn$.
For general $k\geq 3$, $\Delta\in\{0,1\}$,
including the case $X\not\in E, \Delta = 1$, 
we have TMs of size $(35+3k+\Delta(7+k),2k+2-\Delta)$.

The (47,10) implementation immediately yields implementations with sizes (2,2450) and (3,1440) as well as  (922,2) and (353,3), where we use the somewhat crude upper bounds from Theorems~\ref{thm:Q2}, \ref{thm:Q3} and \ref{thm:Eb}.


\begin{theorem} \label{thm:cases}
  There are TMs with the following sizes to compute $n(3)$ and $n(4)$, respectively$:$

  {\rm
  \bt{cccll}
  $n$&$m$&Case&$k=3$&$k=4$\\
  \cline{1-5}
  $($2,&$\bullet)$&C1&TM(2,1840)&TM(2,2450)\\
  $($3,&$\bullet)$&C2&TM(3,1080)&TM(3,1440)\\
  $($9,&$\bullet)$&C3&TM(9,800)&TM(9,1030)\\
  $(\bullet$,&$2k+2)$&A&TM(44,8)&TM(47,10)\\
  $(\bullet$,&$2k+1)$&A&TM(54,7)&TM(58,9)\\
  $(\bullet$,&$4)$&--/A&---&TM(160,4)\\
  $(\bullet$,&$3)$&A/B&TM(155,3)&TM(353,3)\\
  $(\bullet$,&$2)$&A/B&TM(276,2)&TM(922,2)\\
  \et
  }
  
\end{theorem}

\begin{proof}
  By construction in case A.
  By applying Theorem~\ref{thm:Eb} in case B, Theo-rem~\ref{thm:Q2} in case C1, Theorem~\ref{thm:Q3} in case C2, and Theorem~\ref{thm:Q2b} in case C3.
\end{proof}

\section{Infeasable Busy Beaver contests}

Whenever a $(n,m)$ contest lies beyond an $(n',m')$ implementation of $n(3)$ or $n(4)$, {\it i.e.} $n\geq n', m\geq m'$, it can safely be considered infeasable.
On the other hand, all {\tt BB(1,m)} and {\tt BB(n,1)} contests are trivial.

That leaves us with the interesting cases, {\it i.e.} both feasable and non-trivial, as collected in the left column, $n(3)$, of Figure~\ref{fig:BBcases}, where we use only the cases from Theorem~\ref{thm:cases} to interpolate.
The right column, $n(4)$, already goes beyond feasability.

In summary,
all but 37022 non-trivial cases $(n,m)$ lead to ${\tt BB(n,m)} > A(7198,158386)$ and
all but 51671 cases even have ${\tt BB(n,m)} > A^{(A(178195))}(1)$.

\newpage
\section*{Conclusion}

We have shown that at most 51671 {\tt BB(n,m)} values are below $A^{(A(1178195))}(1)$.
This gives  a third description of the difficulty of the Busy Beaver problem:

\begin{figure}[b]
  \centering 
    \bt{llr}
    n&m&number\\
    \cline{1-3}
    2&2..1839&1838\\
    3..8&2..1079&6468\\
    9..43&2..799&27930\\
    44..53&2..7&60\\
    54..154&2..6&505\\
    155..275&2&221\\
&&\\
    \cline{1-3}
    \multicolumn{2}{l}{Total:}&37022\\
    \et
    \bt{|llr}
    n&m&number\\
    \cline{1-3}
    2&2..2449&2448\\
    3..8&2..1439&8628\\
    9..43&2..1029&39064\\
    47..57&2..7&66\\
    58..159&2..6&510\\
    160..352&2..3&386\\
    353..921&2&569\\
    \cline{1-3}
    \multicolumn{2}{|l}{Total:}&51671\\
    \et
    \caption{Interesting {\tt BB} cases, as bounded by $n(3)$ (left) and $n(4)$ (right)} \label{fig:BBcases}
\end{figure}

{\bf Goldbach}
  {\sc CodeGolfAddict} \cite{CGA}
  gives a (27, 2) implementation to search for a counterexample for Goldbach's conjecture that every even number is the sum of two primes.
  The value {\tt BB(27,2)} thus depends on a mathemati\-cian's (number theory, {\it not} TCS) success to solve Goldbach's conjecture.

{\bf Long Finite Sequences}
Our contribution shows that {\tt BB(44,8)} and  \linebreak
{\tt BB(276,2)} are computationally infeasable, lying beyond $A(7198,158386)$.

{\bf ZFC}
Yedidia and Aaronson \cite{Yedidia} have given a (7910,2) TM that checks ZFC for congruency  --- going back to another work of Friedman.
Hence, resolving {\tt BB(7910,2)} is outside the scope of ZFC.
Stefan O'Rear has a (748,2) implementation \cite{ORear}.

Note that Friedman's problem is a definite lower bound, while the other two problems can be put aside as holdouts that might never halt and thus not enter the {\tt BB} contest, when you have confidence in Christian Goldbach, Ernst Zermelo and Adolf Fraenkel.

We furthermore have described algorithms to obtain, for any given TM, Turing machines with either state count or symbol count freely selectable down to the value 2.

\vfill

\noindent{\tt (C++ implementations available upon request from the first author)}

\end{document}